\documentclass[12pt,article]{amsart}

\usepackage[margin=1in]{geometry}
\usepackage{amsmath, amssymb,amsthm,url,mathrsfs,graphicx,amscd,amsfonts,bbm}
\usepackage[mathscr]{eucal}
\usepackage{dsfont}
\usepackage{mathtools}
\usepackage{hyperref}

\usepackage[utf8]{inputenc}
\usepackage{graphicx,todonotes,hyperref}
\usepackage{epstopdf}
\usepackage{inputenc}
\usepackage{enumitem}
\usepackage{amsmath}
\usepackage{graphicx,todonotes,hyperref}
\usepackage{xcolor}
\usepackage{amssymb}

\newtheorem*{acknowledgements*}{Acknowledgements}

\allowdisplaybreaks
\newtheorem{theorem}{Theorem}[section]
\newtheorem{lemma}[theorem]{Lemma}

\newtheorem*{definition}{Definition}
\newtheorem*{theorem*}{Theorem}
\theoremstyle{remark}

\numberwithin{equation}{section}

\renewcommand{\pmod}[1]{\left(\mathrm{mod}\,#1\right)}

\newcommand{\R}{\mathbb{R}}
\newcommand{\Z}{\mathbb{Z}}

\newcommand{\calP}{\mathcal{P}}
\newcommand{\calL}{\mathcal{L}}

\begin{document}
	
	\title[Restriction estimates with sifted integers]{Restriction estimates with sifted integers}
	
	\author{Tanmoy Bera}
	\address{$^{1}$Harish-Chandra Research Institute, Chhatnag Road, Jhunsi, Prayagraj, 211019 India\\
		$^{2}$Homi Bhabha National Institute, Training School Complex, Anushakti Nagar, Mumbai, 400094 India}
	\author{G. K. Viswanadham}
	\address{$^{3}$Department of Mathematical Sciences, IISER Berhampur, Ganjam, Odisha, India 760003}
	\email{Tanmoy Bera$^{1,2}$: tanmoybera@hri.res.in}
	\email{G. K. Viswanadham$^3$: viswanadh@iiserbpr.ac.in}
	\dedicatory{Dedicated to Prof. Ramar{\'e} on the occasion of his 60th bithday}

	
	\begin{abstract}
		Let $\mathcal{P}$ be a subset of primes and for each prime $p\in \mathcal{P}$, consider a subset $\mathcal{L}_p$ of $\mathbb{Z}/p\mathbb{Z}$. We provide restriction estimates with integers $\leq N$ sifted by $(\mathcal{L}_p)_{\substack{p\leq z\\ p\in \mathcal{P}}}$. This generalizes a result of Green-Tao \cite{Green-Tao*06} on the restriction estimates. 
	\end{abstract}
	\maketitle
	\section{Introduction}
	For a given real number $\alpha\in [0,1]$ and positive integer $N$, define
	\begin{equation}
		\label{S-alpha,N}
		S(\alpha, N)=\sum_{p\leq N} a_p e(p\alpha)\; ,
	\end{equation}
	where the sum runs over  primes $p\leq N$ and $(a_p)_{p\leq N}$ is a sequence of complex numbers. Bourgain \cite{Bourgain*89} obtained an upper bound for the $L_{\ell}$-norm of $S(\alpha,N)$ for any $\ell>2$. More precisely, he obtained that
	\begin{equation}
		\label{restriction}
		\left(\int_{0}^1|S(\alpha,N)|^\ell d\alpha\right)^{1/\ell}\ll_{\ell}  N^{-1/\ell} \left(\frac{N}{\log N}\sum_{p\leq N}|a_p|^2\right)^{1/2}\; .
	\end{equation}
	The above bound can be seen to be optimal, up to the constant. The estimates of this type are called {\em restriction estimates}.  An alternative proof of the result of Bourgain \cite{Bourgain*89} is given by Green (see~\cite[Theorem~2.1]{Green*03}) and used it to show that the primes enjoy the Hardy-Littlewood majorant property (see Theorem 1.5 of \cite{Green*03})in the sense of~\cite{Green-Ruzsa 2004}:
	
	\begin{equation}
		\label{majorant property}
		\int_0^1\left|\sum_{p\leq N} a_pe(p\alpha)\right|^\ell d\alpha\ll_\ell \int_{0}^1\left|\sum_{p\leq N}e(p\alpha)\right|^\ell d\alpha\; 
	\end{equation}
	for any sequence $(a_p)$ of complex numbers with $|a_p|\leq 1$ and $\ell\geq 2$.

	\vspace{2mm}
	\noindent
	These results were generalized to a wider range of arithmetic sets,  instead of primes in \eqref{S-alpha,N},  by 
	Green-Tao \cite{Green-Tao*06}, using the enveloping sieve developed by Ramar\'e \cite{Ramare*95}, and Ramar\'e and Ruzsa \cite{Ramare-ruzsa}. 
	
	\vspace{2mm}
	\noindent
	In the other direction, restriction estimates with smooth numbers are obtained by Harper (see Theorem 2, \cite{Harper*16}).

	\vspace{2mm}
	\noindent
	The restriction estimate \eqref{restriction} fails at the endpoint $\ell=2$, which can be seen using Parseval's identity. Ramar\'e \cite{ramare24} provided a smooth transition from the case $\ell=2$ to the case $\ell>2$ and obtained explicit values for the constants in the estimates \eqref{restriction} and  \eqref{majorant property}. He showed that the constant in \eqref{majorant property} can be chosen to be absolute, namely $10^5$ when $N\geq 10^6$. He also provided restriction estimates with primes in intervals. In a recent preprint, Ramar\'e \cite{Ramare*25} improved these explicit constants.  
	
	\vspace{2mm}
	\noindent
	The applications of restriction estimates to number theory were initiated by Bourgain \cite{Bourgain*89}. Later,  these estimates were found to be very useful in number theory, for example, see \cite{Green*03, Harper*16}.
	
	\vspace{2mm}
	\noindent
	The aim of this article is to obtain restriction estimates with sifted sequences, which we will describe below.
	
	\vspace{2mm}
	\noindent 
	Let $\mathcal{P}$ be a non-empty subset of primes, and $2\leq y_0\leq y$ be two real numbers. We define  
	$$\mathcal{P}(y)=\prod_{\substack{p<y\\ p\in\mathcal{P}}}p,\text{ and } \mathcal{P}(y_0,y)=\mathcal{P}(y)/\mathcal{P}(y_0).$$
	For any $p\in \mathcal{P}$, let $\mathcal{L}_p$ be a non-empty proper subset of $\Z/p\Z$ with cardinality $\lambda(p)$.That is, $0<\lambda(p)<p$. We assume that there exists a positive constant $M$ such that 
	\begin{equation}
		\tag{H1}\label{H1}
		\lambda(p)\leq M p^{1/4} \text{ for all } p\in\mathcal{P}\;.
	\end{equation} 
	For primes $p\notin\mathcal{P}$ we take $\mathcal{L}_p=\emptyset$. We then extend it for any squarefree integer $\ell$ by defining 
	\[\mathcal{L}_\ell=\prod_{p|\ell}\mathcal{L}_p\subseteq \mathbb{Z}/\ell\mathbb{Z}\,,\; \; \text{ and }  \; \; \; \lambda(\ell)=\prod_{p|\ell}\lambda(p)\;.\]
	For any $y\geq 2$, we denote
	\[V(y):=\prod_{\substack{p<y\\ p\in \mathcal{P}}}\left(1-\frac{\lambda(p)}{p}\right)^{-1}\;.\]
	We also assume that there exist positive constants $T, \kappa$ such that for any $y\geq 2$
	\begin{equation}
		\tag{H2}\label{H2}
		V(y)\leq T \log^\kappa y\;.
	\end{equation}
	
	\vspace{2mm}
	\noindent
	Let $h$ be the non-negative multiplicative function supported on square-free integers defined by  $h(p)=\frac{\lambda(p)}{p-\lambda(p)}$ for any $p\in \mathcal{P}$, and $0$ if $p\notin \mathcal{P}$.  
	
	\vspace{2mm}
	\noindent
	For any positive integer $d$, and real numbers $y\geq 1,\; z_0\geq 2$, let 
	\begin{align*}
		G_d(y;z_0):= \sum_{\substack{\ell\leq y\\ (\ell,\;  d\mathcal{P}(z_0))=1}} h(\ell)\; ,
	\end{align*}
	and let $G_1(y,z_0):=G(y,z_0)$, $G_1(y,2):=G(y)$.
	
	\vspace{2mm}
	\noindent
	Let $z$ be a real number. For any positive integer $N,$  define
	\[\mathcal{S}(N):= \mathcal{S}(N,\mathcal{P}, z):= \{n\leq N: n\pmod{p}\notin \mathcal{L}_p \text{ for any } \; p\in \mathcal{P} \text{ and } p<z\}\;.\]
	The main theorem of this article is the following: 
	\begin{theorem}
		\label{coro}
		Let $z\leq N^{1/4}$ and let $\mathcal{S}(N)$ be as defined above. Then for any sequence of complex numbers $(a_n)_{n\leq N}$ and $\ell>2$, we have 
		\begin{equation}
			\label{rest-est}
			\int_0^1\left|\sum_{n\in \mathcal{S}(N)}a_ne(n\alpha)\right|^{\ell}d\alpha\leq \frac{C}{N} \left(\frac{2N}{G(z)}\sum_{n\in \mathcal{S}(N)}|a_n|^2\right)^{\ell/2}\; ,
		\end{equation}
		where $C=C(\kappa, \ell, T, M)$ is a constant  that can be computed explicitly. 
	\end{theorem}

	\vspace{2mm}
	\noindent
	As one can see by Parseval identity, Equation~\eqref{rest-est} doesn't hold for $\ell=2$. The exact dependency of the constant on $\ell$ is given in \eqref{c_1} and \eqref{c_2},  and this provides a smooth transition from $\ell>2$  to $\ell=2$.  This striking feature comes from the method of Ramar\'e \cite{ramare24}. 
	
	\vspace{2mm}
	\noindent
	\begin{definition}
		Let $\delta>0$ be a real number. A finite set $\mathcal{B}\subseteq \mathbb{R}/\mathbb{Z}$ is said to be a $\delta$-well spaced set if 
		\[\|x_i-x_j\|\geq \delta \text{ for all } x_i,x_j\in \mathcal{B}\;,\]
		where $\|x\|=\min_{k\in\Z}|x+k|$ denotes the distance from  the closest integer to $x$.
	\end{definition}
	
	The above theorem is a simple application of the following theorem:
	
	\begin{theorem}
		\label{rest-1}
		Let $z\leq N^{1/4}$ and let $\mathcal{S}(N)$ be as defined above and $\mathcal{B}$ be a $\delta$-well spaced subset of $\R/\Z.$ Then for any complex numbers $(a_n)_{n\leq N}$ and $\ell>2$,  we have
		\begin{equation}
			\sum_{b\in \mathcal{B}}\left|\sum_{n\in \mathcal{S}(N)}a_ne(nb)\right|^{\ell}\leq C\left(\frac{N+\delta^{-1}}{G(z)}\sum_{n\in \mathcal{S}(N)}|a_n|^2\right)^{\ell/2}\; ,
		\end{equation}
		where $C=C(\kappa, \ell, T, M)$ is a constant  that can be computed explicitly. 
	\end{theorem}  
	\noindent
	By taking $\mathcal{B}=\{\beta+k/N; 0\leq k\leq N-1\}$ and integrating over $\beta$ in $[0,1/N]$, Theorem~\ref{coro} follows.

	\vspace{2mm}
	\noindent
	In the proof of the above theorem, we have provided the exact dependency of the constant on $\ell$ (see \eqref{c_1} and \eqref{c_2}).

	\vspace{2mm}
	\noindent
	To prove Theorem~\ref{rest-1}, we need a dual large sieve type inequality, which we will describe below.

	\vspace{2mm}
	\noindent
	Let  $\mathcal{B}$ be a $\delta$-well spaced subset of $\mathbb{R}/\mathbb{Z}$ and $f$ be a function on $\mathcal{B}$.  In~\cite{ramare24}, Ramar\'e 
	proved that  for any  $N\geq 10^3$, we have 
	\[\sum_{p\leq N}\left|\sum_{b\in \mathcal{B}}f(b)e(bp)\right|^2\leq 280\frac{N+\delta^{-1}}{\log N}\|f\|_2^2\log\left(2\frac{\|f\|^2_1}{\|f\|_2^2}\right)\;.\]
	The constant 280 is further improved to $19$ (for $N\geq 10^4$) by Ramar\'e \cite[Theorem~6.2]{Ramare*25}.
	We need an analogous  estimate with  $\mathcal{S}(N)$ instead of primes $p\leq N$.

	\begin{theorem}\label{dual}
		Let $N$ and $z$ be  positive real numbers such that $z\leq N^{1/4}$, and $\delta>0$. Let $\mathcal{B}$ be a $\delta$-spaced subset of $\mathbb{R}/\mathbb{Z}$. Then for any non-zero function $f$ on $\mathcal{B}$, we have
		\[\sum_{n\in \mathcal{S}(N)}\left|\sum_{b\in \mathcal{B}}f(b)e(nb)\right|^2\leq K \frac{N+\delta^{-1}} {G(z)}\log^\kappa\left(2\frac{\|f\|_1^2}{\|f\|_2^2}\right)\|f\|_2^2\]
		with some constant $K$ that depends only on $T$, $M$ and $\kappa$.
	\end{theorem}
	\vspace{2mm}
	\noindent	
	This result is the heart of our paper. The proof of Theorem~\ref{dual} relies on the enveloping sieve of Ramar\'e and Ruzsa (see Section~4 of \cite{Ramare-ruzsa}) and follows the method of Ramar\'e \cite{ramare24}.   The strategy  is
	not to study the characteristic function of the set $\mathcal{S}(N)$ directly, but rather to construct a majorant $\beta$, which has good Fourier properties, using the enveloping sieve. 
	Using the theorem above, we can prove the following large sieve inequality with sifted integers by the standard argument.
	
	\begin{theorem}\label{dual1}
		\label{large}
		Let $N$ and $z$ be  positive real numbers such that $z\leq N^{1/4}$, and $\delta>0$. Let $\mathcal{B}$ be a $\delta$-spaced subset of $\mathbb{R}/\mathbb{Z}$. Then for any sequence of complex numbers $(a_n)_{n\leq N}$, we have 
		\[
		\sum_{b\in \mathcal{B}}\left|\sum_{n\in \mathcal{S}(N)}a_ne(nb)\right|^2\leq K \frac{(N+\delta^{-1})}{G(z)}\log^{\kappa}(2|\mathcal{B}|) \sum_{n\in \mathcal{S}(N)}|a_n|^2\; .
		\]
	\end{theorem}

	\noindent
	{\em Organization of the paper:} In Section~\ref{nlms},  we provide some necessary lemmas.  In Section~\ref{envsieve}, we discuss the enveloping sieve. In sections \ref{pthm1.3}, \ref{proof-large} and \ref{proof-1.2}, we prove Theorems \ref{dual}, \ref{large} and \ref{rest-1} respectively. We will discuss Hardy-Littlewood majorant property for $\mathcal{S}(N)$ in  Section~\ref{HLMP}. In the last section, we provide some applications of our results.

	\section{Necessary Lemmas}\label{nlms}
	\noindent
	
	\begin{lemma}
		When $d$ is coprime to $\mathcal{P}(z_0)$, we have 
		\[
		G(yd; z_0)\geq \left(\sum_{\delta\mid d} h(\delta)\right) G_d(y; z_0)\geq G(y; z_0)\; .
		\]
	\end{lemma}
	\begin{proof}
		We have 
		\begin{align*}
			G(y;z_0)= \sum_{\delta\mid d}\sum_{\substack{\ell\leq y\\ (\ell, \mathcal{P}(z_0))=1\\ (\ell, d)=\delta}}h(\ell)
			=  \sum_{\delta\mid d} h(\delta)\sum_{\substack{m\leq y/\delta\\ (m, d\mathcal{P}(z_0)=1)}}h(m)
			\geq \left(\sum_{\delta\mid d}h(\delta)\right) G_{d}(y/d, z_0)\; .
		\end{align*}
		Here we used the fact that $h$ is supported on square-free integers.
		Taking $yd$ in place of $y$  gives us the first inequality. On the other hand, 
		\begin{align*}
			G(y, z_0)=  \sum_{\delta\mid d} h(\delta)\sum_{\substack{m\leq y/\delta\\ (m, d\mathcal{P}(z_0))=1}}h(m)
			\leq  \left(\sum_{\delta\mid d}h(\delta)\right)\sum_{\substack{m\leq y\\ (m, d\mathcal{P}(z_0))=1}}h(m)
			= \left(\sum_{\delta\mid d}h(\delta)\right) G_d(y, z_0) \;,
		\end{align*}
		which gives the second inequality. 
	\end{proof}
	
	\begin{lemma}
		For any $z_1\geq 2,$ we have 
		\[
		G(y)\leq \left(\sum_{\delta\mid \mathcal{P}(z_1)}h(\delta)\right) G(y; z_1)\; .
		\]
	\end{lemma}
	\begin{proof}
		This follows immediately by taking $z_0=2$ and $d=\mathcal{P}(z_1)$ in the previous lemma. 
	\end{proof}

	\begin{lemma}
		\label{g(z;z_0)}
		$G(y;z_0)\geq \prod_{\substack{p<z_0\\ p\in \mathcal{P}}}\left(1-\frac{\lambda(p)}{p}\right) G(y)$.
	\end{lemma}
	\begin{proof}
		This follows from the previous lemma after noticing that
		\begin{align*}
			\sum_{\delta\mid \mathcal{P}(z_0)}h(\delta)=&\prod_{\substack{p<z_0\\ p\in \mathcal{P}}}(1+h(p))
			= \prod_{\substack{p<z_0\\ p\in \mathcal{P}}}\frac{p}{p-\lambda(p)}
			=\prod_{\substack{p<z_0\\ p\in \mathcal{P}}}\left(1-\frac{\lambda(p)}{p}\right)^{-1}\; .
		\end{align*}
	\end{proof}
	\begin{lemma}
		\label{lowerGz}
		$G(y)$ satisfies the following:
		\begin{align}\label{G(z) upper bound}
			G(y)\leq \prod_{\substack{p<\leq y\\ p\in \mathcal{P}}} \left(1-\frac{\lambda(p)}{p}\right)^{-1}\;.
		\end{align}
	\end{lemma}
	\begin{proof}
		
		\begin{align*}
			G(y)=&\sum_{\ell\leq y}h(\ell)
			\leq  \prod_{p\leq y} \left(1+h(p)\right)
			= \prod_{p\leq y} \left(1-\frac{\lambda(p)}{p}\right)^{-1}
		\end{align*}
	\end{proof}
	
	\noindent
	We also need the following lemma.
	
	\begin{lemma}\cite[Lemma 3.3]{ramare24}
		\label{ramare24}
		Let $M\in\mathbb{R},$ and $N$, $\delta$ be positive real numbers. There exists a smooth function $\psi$ on $\R$ such that 
		\begin{itemize}
			\item The function $\psi$ is non-negative.
			\item When $t\in[M,M+N],$ we have $\psi(t)\geq 1$.
			\item $\widehat{\psi}(0)=N+\delta^{-1}$.
			\item When $|\alpha|>\delta,$ we have $\widehat{\psi}(\alpha)=0$.
			\item We have $\psi(t)=\mathcal{O}_{M,N,\delta}(1/(1+|t|^2))\;.$
		\end{itemize}
	\end{lemma}
	
	\noindent
	We refer to \cite{Montgomery} for a proof of this lemma.

	\vspace{2mm}
	\noindent
	The following identity will be used in the proof of Theorem~\ref{rest-1}.
	\begin{lemma}[The Carlitz identity]
		\label{Carlitz}
		For any $n\geq 0$,
		\[
		\sum_{k\geq 0}(k+1)^nt^k=\frac{A_n(t)}{(1-t)^{n+1}}\; ,
		\]
		where $A_n$ is the $n$th Eulerian polynomial. 
	\end{lemma}
	\noindent
	For a discussion on Eulerian polynomials and a proof of this result, we refer to Chapter~1 of \cite{Petersen}.
	
	\section{An enveloping Sieve}\label{envsieve}
	\vspace{2mm}
	\noindent
	The results of this section are taken from the paper of Ramar\'e and Ruzsa, see Section~4 of \cite{Ramare-ruzsa}.
	Let $z_0 $ and $z$ be positive real numbers with $z_0\leq z$. Let
	\[\mathcal{S}^{\prime}(N)=\mathcal{S}^{\prime}(N, \mathcal{P}; z_0,z):=\{n\leq N: n\pmod p \notin \mathcal{L}_p \text{ for any } p\in \mathcal{P}\cap [z_0,z)\}\;.\] 
	
	\vspace{2mm}
	\noindent
	For any prime $p\in \mathcal{P}\cap [z_0,z)$, put $\mathcal{K}_p=\mathbb{Z}/p\mathbb{Z}\setminus\mathcal{L}_p$ and let $\mathcal{K}_p=\mathbb{Z}/p\mathbb{Z}$ for other primes.
	
	\vspace{2mm}
	\noindent
	For any square-free integer $q$, let 
	\begin{align}\label{K_q def}
		\mathcal{K}_q=\prod_{p\mid q}\mathcal{K}_p\subseteq \mathbb{Z}/q\mathbb{Z}\;.
	\end{align}
	
	\vspace{2mm}
	\noindent
	Let $h_{[z_0,z]}$ be the multiplicative function supported on squarefree integers defined by $h_{[z_0,z]}(p)=\frac{\lambda(p)}{p-\lambda(p)}$ for $p\in \mathcal{P}\cap[z_0,z)$ and $0$ otherwise. 
	
	\begin{lemma}
		\label{lembeta}
		Let $\beta_{\mathcal{{S}^{\prime}}(N)}$ denote the characteristic function of the set $\mathcal{S'}(N)$. Then we have 
		\begin{equation}
			\label{betaS}
			\beta_{\mathcal{S}^{\prime}(N)}(n)=\sum_{\substack{q\leq z^2\\ q\mid \mathcal{P}(z_0,z)}} \sum_{a\;  \text{mod}^\star q}w_{[z_0,z]}(a/q)e\left(\frac{an}{q}\right)\; ,
		\end{equation}
		where \[ w_{[z_0,z ]}(a/q)=\frac{1}{G(z; z_0)^2}\left(\frac{1}{|\mathcal{K}_q|}\sum_{b\in \mathcal{K}_q}e\left(\frac{ab}{q}\right)\right)\sum_{\substack{k\leq z}}h_{[z_0,z]}(k)\sum_{\substack {r_1,d_1,r_2,d_2\\ r_1d_1, r_2d_2\leq z\\ k\mid (d_1, d_2)\\ q\mid [d_1, d_2]}}\mu(r_1)\mu(r_2)\; .\]
		
	\end{lemma}
	\begin{proof}
		This follows by combining  Equations~4.1.13, 4.1.14, and 4.1.21 of \cite{Ramare-ruzsa}.
		
	\end{proof}
	\begin{lemma}
		\label{w_qbound}
		We have 
		\[
		w_{[z_0,z]}(a/q) \ll \frac{1}{\sqrt{q}} \prod_{\substack{p\in \mathcal{P}\\ p<z_0}}\left(1-\frac{\lambda(p)}{p}\right)^{-1}\frac{1}{G(z)}\; .
		\]
	\end{lemma}
	\begin{proof}
		By Equations 4.1.13 and 4.1.16 of~\cite{Ramare-ruzsa}, we get
		\[
		w_{[z_0,z]}(a/q)\ll \frac{1}{G(z;z_0)}3^{\omega(q)}\frac{1}{|\mathcal{K}_q|}\left|\sum_{b\in \mathcal{K}_q}e\left(\frac{ba}{q}\right)\right|\; .
		\]
		where $\mathcal{K}_q$ is as defined in~\eqref{K_q def}.
		By a simple application of the Chinese remainder theorem, we have
		\[
		\frac{1}{|\mathcal{K}_q|}\sum_{b\in \mathcal{K}_q}e\left(\frac{ba}{q}\right)\leq \prod_{p\mid q}\frac{\lambda(p)}{p-\lambda(p)}\; .
		\]
		Using the hypothesis \eqref{H1}, i.e.  $\lambda(p)\leq M p^{1/4}$, we get that
		\[3^{\omega(q)}\prod_{p\mid q}\frac{\lambda(p)}{p-\lambda(p)}\ll_{M} \frac{1}{\sqrt{q}}\;. 
		\]
		This inequality follows as $\frac{3\lambda(p)}{p-\lambda(p)}\leq \frac{1}{\sqrt{p}}$ holds except for $\ll_{M}1$ primes.
		
		\vspace{2mm}
		\noindent
		The lemma follows by combining the above inequality with Lemma~\ref{g(z;z_0)}.
	\end{proof}
	
	\section{Proof of Theorem~\ref{dual}}\label{pthm1.3}
	
	When $2^\kappa T\log^{\kappa}\left(2\|f\|_1^2/\|f\|_2^2\right)\geq V(z)$, we use the dual large sieve inequality (see \cite{Montgomery}) to get 
	\begin{align*}
		\sum_{n\in \mathcal{S}(N)}\left|\sum_{b\in \mathcal{B}}f(b)e(bn)\right|^2\leq& (N+\delta^{-1})\|f\|_2^{2}\\
		\leq & 2^\kappa T(N+\delta^{-1})\|f\|_2^{2}\log^{\kappa}\left(2\frac{\|f\|_1^2}{\|f\|_2^2}\right) V(z)^{-1}\; .
	\end{align*}
	Therefore, imposing~\eqref{G(z) upper bound} here we obtain
	\begin{align}\label{case 1}
		\sum_{n\in \mathcal{S}(N)}\left|\sum_{b\in \mathcal{B}}f(b)e(bn)\right|^2
		\leq2^\kappa T(N+\delta^{-1})\|f\|_2^{2}\log^{\kappa}\left(2\frac{\|f\|_1^2}{\|f\|_2^2}\right) G(z)^{-1}\; .
	\end{align}
	
	\noindent
	We assume now onwards that 
	\begin{equation}
		\label{small}
		2^\kappa \log^{\kappa}\left(2\|f\|_1^2/\|f\|_2^2\right)\leq \frac{1}{T}V(z)\;.
	\end{equation}
	
	\noindent
	We bound the contribution from small integers trivially:
	\[
	\sum_{\substack{1\leq n\leq z\\ n\in \mathcal{S}(N)}}\left|\sum_{b\in \mathcal{B}}f(b)e(bn)\right|^2\leq z \|f\|_1^2\; .
	\]
	From \eqref{small} and the assumption~\eqref{H2} on $V(z)$, we get $\|f\|_1^2\leq \frac{1}{2}z^{1/2}\|f\|_2^2$. This gives us 
	\[
	\sum_{\substack{1\leq n\leq z\\ n\in \mathcal{S}(N)}}\left|\sum_{b\in \mathcal{B}}f(b)e(bn)\right|^2\leq 2^{-1}z^{3/2}\|f\|_2^2\; .
	\]
	Since $\log\left(2\frac{\|f\|_1^2}{\|f\|_2^2}\right)\geq \log(2),$ by the hypothesis~\eqref{H2} $V(z)\leq T\log^\kappa z$, we get that 
	\begin{align}
		\label{small-n}
		\sum_{\substack{1\leq n\leq z\\ n\in \mathcal{S}(N)}}\left|\sum_{b\in \mathcal{B}}f(b)e(bn)\right|^2&\leq\frac{T}{2\log^\kappa 2} z^{3/2} \log^\kappa(z)\|f\|_2^2\log^\kappa\left(2\frac{\|f\|_1^2}{\|f\|_2^2}\right) V(z)^{-1}\nonumber\\
		&\leq\frac{T}{2\log^\kappa 2} \left(N+\delta^{-1}\right)\|f\|_2^2\log^\kappa\left(2\frac{\|f\|_1^2}{\|f\|_2^2}\right) G(z)^{-1}\; .
	\end{align}
	Here, in the last line, we used that $z\leq N^{1/4}$ and~\eqref{G(z) upper bound}.

	\vspace{2mm}
	\noindent
	Put 
	\begin{equation}
		\label{z0}
		z_0=\left(2\frac{\|f\|_1^2}{\|f\|_2^2}\right)^2\; .
	\end{equation}
	Then we have $z_0\leq z$. Let 
	\begin{align}\label{W def}
		W=\sum_{\substack{z<n\leq N\\ n\in \mathcal{S}^{\prime}(N)}}\left|\sum_{b\in\mathcal{B}}f(b)e(nb)\right|^2\;.
	\end{align}
	Clearly 
	\[ 
	\sum_{\substack{z\leq n\leq N\\ n\in \mathcal{S}(N)}}\left|\sum_{b\in \mathcal{B}}f(b)e(bn)\right|^2 \leq W\; .
	\]
	
	\noindent	 We bound the characteristic function of $\mathcal{S}^{\prime}(N)\cap(z,N]$ by $\beta_{\mathcal{S}^{\prime}(N)}$ and use the expression \eqref{betaS}. 
	
	\vspace{2mm}
	\noindent
	Let $\delta_1=\min(\delta,\frac{1}{2z^4})$. Majorizing  the characteristic function of the interval $[z, N]$ by the smooth function  $\psi$ whose Fourier transform supported on $[-\delta_1,
	\delta_1]$ (see Lemma~\ref{ramare24}) gives us
	\[
	W\leq\sum_{\substack{q\leq z^2\\ q\mid \mathcal{P}(z_0,z)}}\sum_{a\mod^* q}w_{[z_0,z]}(a/q) \sum_{b_1,b_2\in \mathcal{B}}|f(b_1)\overline{f(b_2)}|\sum_{n\in \mathbb{Z}} \psi(n) e\left(\left(b_1-b_2+\frac{a}{q}\right)n\right)\; .
	\]
	\noindent
	The innermost sum, by the Poisson summation formula, is $\sum_{m\in \mathbb{Z}}\widehat{\psi}\left(b_1-b_2+\frac{a}{q}+m\right)$. Since $\widehat{\psi}$ is supported on $[-\delta_1, \delta_1]$, the last sum over $n, b_1, b_2$ is 
	\[
	\leq (N+\delta_1^{-1}) \sum_{b_1,b_2} |f(b_1)f(b_2)|\# \{a/q:\|b_1-b_2+a/q\|<\delta_1\}\;.
	\]
	
	\noindent
	Now we split the sum over $q$ into two parts, i.e.  $q<z_0$ and $q\geq z_0$: 
	\[W=W(q<z_0)+W(q\geq z_0)\;.\]
	First, we consider the sum $W(q\geq z_0)$. 
	Since $\delta_1\leq \frac{1}{2z^4}$, there will be at most one $\frac{a}{q}$ which satisfy the above inequality. Hence
	\[
	W(q\geq z_0)\leq (N+\delta_1^{-1}) \|f\|_1^2|w_{[z_0,z]}(a/q)|
	\]
	for some $a \mod^{\star} ( q),\;  q\leq z^2$. By Lemma~\ref{w_qbound}, we get that
	\begin{align*}
		W(q\geq z_0)\ll (N+\delta_1^{-1}) \|f\|_1^2 \frac{1}{\sqrt{q}} \prod_{\substack{p\in \mathcal{P}\\ p< z_0}}\left(1-\frac{\lambda(p)}{p}\right)^{-1}\frac{1}{G(z)}\; .
	\end{align*}
	Since $q\geq z_0,$ we obtain
	\begin{align*}
		W(q\geq z_0) \ll (N+\delta_1^{-1}) \|f\|_1^2 \frac{1}{\sqrt{z_0}}V(z_0)G(z)^{-1}\; .
	\end{align*}
	Using $\sqrt{z_0}=2\|f\|_1^2/\|f\|_2^2$ and the assumption~\eqref{H2} on $V(z_0)$, we will reach
	\begin{align}
		\label{big-q}
		W(q\geq z_0)\ll&  (N+\delta_1^{-1}) \log^\kappa(z_0)\|f\|_2^2 G(z)^{-1}\nonumber\\
		\ll &(N+\delta_1^{-1}) \log^\kappa\left(2\frac{\|f\|_1^2}{\|f\|_2^2}\right)\|f\|_2^2 G(z)^{-1}\; .
	\end{align}
	
	\vspace{2mm}
	\noindent
	On the other hand, when $q\leq z_0$, since  $q\mid \calP(z_0,z)$, only $q=1$ contributes to the sum $W(q\leq z_0)$. Thus
	\begin{align*}
		W(q\leq z_0) \leq & (N+\delta_1^{-1})\|f\|_2^2 w_{[z_0, z]}(1)\\
		\ll & (N+\delta_1^{-1})\|f\|_2^2V(z_0)G(z)^{-1}\;,
	\end{align*}
	where we have used Lemma~\ref{w_qbound} for the last inequality.  Our assumption~\eqref{H2} on $V$ and the choice of $z_0$ give us 
	\begin{align}
		\label{small-q}
		W(q\leq z_0)&\ll   (N+\delta_1^{-1})\|f\|_2^2 \log^\kappa(z_0)G(z)^{-1}\nonumber\\
		&\ll  (N+\delta_1^{-1})\|f\|_2^2 \log^\kappa\left(2\frac{\|f\|_1^2}{\|f\|_2^2}\right)G(z)^{-1}\; .
	\end{align}
	Since $z\leq N^{1/4},$ by definition of $\delta_1$ we get $N+\delta_1^{-1}\ll N+\delta^{-1}.$ Therefore, combining \eqref{big-q}, and \eqref{small-q} we obtain
	\begin{align}
		\label{large-n} 
		\sum_{\substack{z\leq n\leq N\\ n\in \mathcal{S}(N)}}\left|\sum_{b\in \mathcal{B}}f(b)e(bn)\right|^2 \ll(N+\delta^{-1})\|f\|_2^2 \log^\kappa\left(2\frac{\|f\|_1^2}{\|f\|_2^2}\right)G(z)^{-1}\; .
	\end{align}
	Theorem~\ref{dual} follows by combining \eqref{case 1}, \eqref{small-n} and \eqref{large-n}. 
	
	\section{Proof of Theorem~\ref{large}}
	\label{proof-large}
	\noindent
	For this section, let us set $I(b)=\sum_{n\in \mathcal{S}(N)}a_ne(nb)$.
	We have
	\[
	\sum_{b\in \mathcal{B}}\left|\sum_{n\in \mathcal{S}(N)}a_ne(nb)\right|^2= \sum_{b\in \mathcal{B}}\sum_{n\in \mathcal{S}(N)}a_ne(nb)\overline{I(b)}= \sum_{n\in \mathcal{S}(N)}a_n\sum_{b\in \mathcal{B}}\overline{I(b)}e(nb)\; .
	\]
	Using the Cauchy-Schwarz inequality, we get
	\[
	\left(\sum_{b\in \mathcal{B}}\left|\sum_{n\in \mathcal{S}(N)}a_ne(nb)\right|^2\right)^2\leq \sum_{n\in \mathcal{S}(N)}|a_n|^2\sum_{n\in \mathcal{S}(N)}\left|\sum_{b\in \mathcal{B}}\overline{{I}(b)}e(nb)\right|^2\; .
	\]
	Applying Theorem~\ref{dual} to the second sum on the right-hand side gives us
	
	\[
	\left(\sum_{b\in \mathcal{B}}\left|\sum_{n\in \mathcal{S}(N)}a_ne(nb)\right|^2\right)^2\leq K \frac{(N+\delta^{-1})}{G(z)} \left(\sum_{n\in \mathcal{S}(N)}|a_n|^2\right)  \|I\|_2^2 \log^\kappa\left(2\frac{\|I\|_1^2}{\|I\|_2^2}\right)   \;,
	\]
	where $I=(I(b))_{b\in \mathcal{B}}.$ The theorem follows swiftly by noticing that the left-hand side is $\|I\|_2^4$ and $\|I\|_1^2\leq |\mathcal{B}|\|I\|_2^2\;.$

	\section{Proof of Theorem~\ref{rest-1}}
	\label{proof-1.2}
	\begin{lemma}
		\label{y-to-t}
		Suppose  $\frac{y}{\log^\kappa y}\leq t$ for $y\geq 2$ and $t\geq e$. Then $y\leq c(\kappa)t\log^\kappa t$, where $c(\kappa)=(\kappa^2+5)^{\kappa}$.
	\end{lemma}
	\begin{proof}
		The inequality is obvious for $y\leq e^{\kappa}$. 
		Define $f(y)=\frac{y}{\log^\kappa y}$. This is a non-decreasing function for $y\geq e^{\kappa}$. 
		
		\vspace{2mm}
		\noindent
		One can easily check that $f(c(\kappa)t\log^\kappa t)\geq t\geq f(y)$. Hence $y\leq c(\kappa)t\log^\kappa t$.
	\end{proof}
	
	\noindent 
	{\em{Proof of Theorem~\ref{rest-1}}}. 
	This theorem can be proved following the proof of Theorem 1.2 of Ramar\'e \cite{ramare24}. 
	Let $U=\left(2\frac{N+\delta^{-1}}{G(z)}\sum_{n\in \mathcal{S}(N)}|a_n|^2\right)^{1/2}$.

	For any $\xi>0$, define
	\[
	\mathcal{B}_{\xi}=\left\{b\in \mathcal{B}: |\sum_{n\in \mathcal{S}(N)}a_ne(nb)|\geq \xi U\right\}\;.
	\] By Cauchy-Schwarz inequality and the arithmetic large sieve inequality, we have 
	\[
	\left|\sum_{n\in \mathcal{S}(N)}a_ne(bn)\right|^2\leq \sum_{n\in \mathcal{S}(N)}|a_n|^2 \sum_{n\in \mathcal{S}(N)}1\leq \frac{N+z^2}{G(z)}\sum_{n\in \mathcal{S}(N)}|a_n|^2\; .
	\]

	\noindent 
	Since $N+z^2\leq 2(N+\delta^{-1})$, we can take $\xi\leq 1$, as $\mathcal{B}_{\xi}=\emptyset$ if $\xi>1.$
	
	\noindent
	Consider 
	\[
	\Gamma(\xi)=\sum_{b\in \mathcal{B}_{\xi}}\left|\sum_{n\in \mathcal{S}(N)}a_ne(bn)\right|\; .
	\]
	For some $g(b)\in \mathbb{C}$ with $|g(b)|=1$, we have 
	\[
	\Gamma(\xi)=\sum_{n\in \mathcal{S}(N)} a_n\sum_{b\in \mathcal{B}_{\xi}}g(b)e(bn)\; .
	\]
	By Cauchy-Schwarz inequality and Theorem~\ref{dual}, we obtain
	\[
	\Gamma(\xi)^2\leq \sum_{n\in \mathcal{S}(N)}|a_n|^2 \sum_{n\in \mathcal{S}(N)}\left|\sum_{b\in \mathcal{B}_\xi} g(b)e(bn)\right|^2\leq K (U^2/2) |\mathcal{B}_{\xi}|\log^\kappa(2|\mathcal{B}_{\xi}|)\; .
	\]
	On the other hand, we also have
	\[
	\Gamma(\xi)^2\geq \xi^2U^2 |\mathcal{B}_{\xi}|^2\; .
	\]
	Combining these two estimates gives us
	\[
	\frac{2|\mathcal{B}_{\xi}|}{\log^\kappa(2|\mathcal{B}_{\xi}|)}\leq \frac{K}{\xi^2}\; .
	\]
	From Lemma~\ref{y-to-t}, we get
	\begin{equation}
		\label{B-xi-bound}
		|\mathcal{B}_{\xi}|\leq c(\kappa) \frac{K}{\xi^2}\log^\kappa(K/\xi^2)\;.
	\end{equation}
	Let $\gamma>1$ be a real number which we will choose later, and let $\xi_j=\frac{1}{\gamma^j}$. Then
	\[
	\frac{1}{U^{\ell}}\sum_{b\in \mathcal{B}}\left|\sum_{n\in \mathcal{S}(N)}a_n e(bn)\right|^{\ell}\leq |\mathcal{B}_{\xi_0}|+\sum_{j\geq 1}\frac{1}{\gamma^{j\ell}}(|\mathcal{B}_{\xi_j}|-|\mathcal{B}_{\xi_{j-1}}|)\leq \sum_{j\geq 0}\left(\frac{1}{\gamma^{j\ell}}-\frac{1}{\gamma^{(j+1)\ell}}\right)|\mathcal{B}_{\xi_j}|\;.
	\]
	By \eqref{B-xi-bound} we obtain 
	\[
	\frac{1}{U^{\ell}}\sum_{b\in \mathcal{B}}\left|\sum_{n\in \mathcal{S}(N)}a_n e(bn)\right|^{\ell}\leq c(\kappa)K\sum_{j\geq 0}\left(\frac{1}{\gamma^{j\ell}}-\frac{1}{\gamma^{(j+1)\ell}}\right)\gamma^{2j}\log^{\kappa}(K\gamma^{2j})\; .
	\]
	Using the inequality $(a+b)^t\leq (2^{t-1}+1)(a^t+b^t)$, we get
	\begin{align*}
		\frac{1}{U^{\ell}}\sum_{b\in \mathcal{B}}\left|\sum_{n\in \mathcal{S}}a_n e(bn)\right|^{\ell}\leq &c(\kappa)K(2^{\kappa-1}+1)(1-\gamma^{-\ell})\sum_{j\geq 0}\frac{\log^{\kappa}(K)+(2j)^{\kappa}\log^{\kappa}(\gamma)}{\gamma^{j(\ell-2)}}\; .
	\end{align*}
	
	\noindent
	Write $c(\kappa, K)=c(\kappa)K(2^{\kappa-1}+1)$. When $\ell\geq 3$, we take $\gamma=2$. This gives 
	\[
	\frac{1}{U^{\ell}}\sum_{b\in \mathcal{B}}\left|\sum_{n\in \mathcal{S}}a_n e(bn)\right|^{\ell}\leq  c(\kappa,K)\left(\frac{\log^\kappa(K)}{1-2^{2-\ell}}+(2\log(2))^\kappa\sum_{j\geq 0} \frac{j^\kappa}{\gamma^{j(\ell-2)}}\right)\; .
	\]
	We have, by Lemma~\ref{Carlitz}, $\sum_{j\geq 0} \frac{j^{\kappa}}{\gamma^{j(\ell-2)}}\leq \sum_{j\geq 0} \frac{j^{[\kappa]+1}}{\gamma^{j(\ell-2)}}=\gamma^{2-\ell} \frac{A_{[\kappa]+1}(\gamma^{2-\ell})}{(1-\gamma^{2-\ell})^{[\kappa]+2}}$, where $A_n$ is the $n$-th Euler Polynomial. Hence

	\begin{equation}
		\label{c_1}
		\frac{1}{U^{\ell}}\sum_{b\in \mathcal{B}}\left|\sum_{n\in \mathcal{S}}a_n e(bn)\right|^{\ell}\leq c(\kappa,K)\left(\frac{\log^\kappa(K)}{1-2^{2-\ell}}+(2\log(2))^\kappa 2^{2-\ell} \frac{A_{[\kappa]+1}(2^{2-\ell})}{(1-2^{2-\ell})^{[\kappa]+2}}\; \right)\; .
	\end{equation}

	When $\ell\in (2,3)$, we take $\gamma=e^{\frac{1}{\ell-2}}$. This gives us 
	\begin{align}
		\label{c_2}
		\frac{1}{U^{\ell}}\sum_{b\in \mathcal{B}}\left|\sum_{n\in \mathcal{S}}a_n e(bn)\right|^{\ell}\leq &c(\kappa, K)\left(\frac{\log^\kappa(K)}{1-\gamma^{2-\ell}}+(2\log(\gamma))^\kappa \gamma^{2-\ell} \frac{A_{[\kappa]+1}(\gamma^{2-\ell})}{(1-\gamma^{2-\ell})^{[\kappa]+2}}\right)\nonumber\\
		\leq & c(\kappa, K)\left(\frac{\log^\kappa(K)e}{e-1}+(\frac{2}{\ell-2})^\kappa\frac{e^{[\kappa]+1}}{(e-1)^{[\kappa]+2}}A_{[\kappa]+1}(1/e)\right)\; .
	\end{align}

	\section{Hardy-Littlewood majorant property}\label{HLMP}
	Let $\Lambda\subseteq\{1,\dots,N\}$ be a subset and $(a_n)$ be a sequence of complex numbers with $|a_n|\leq1$ for all $n.$ It was observed by Hardy and Littlewood that
	\[\int_0^1\left|\sum_{n\in\Lambda}a_ne(n\alpha)\right|^\ell\leq\int_0^1\left|\sum_{n\in\Lambda}e(n\alpha)\right|^\ell\]
	for any even integer $\ell$. For other values of $\ell>2$, they asked whether anything like the above inequality is true. When $\Lambda$ is the set of all primes in $[1,N]$, Green \cite{Green*03} has shown that there exists a constant $c(\ell)$ that depends only on $\ell$ such that 
	\[
	\int_{0}^1\left|\sum_{p\leq N}a_pe(p\alpha)\right|^\ell d\alpha\leq c(\ell) \int_{0}^1\left|\sum_{p\leq N}e(p\alpha)\right|^\ell d\alpha
	\]
	Now we consider this Hardy-Littlewood majorant problem with $\Lambda=\mathcal{S}(N)$.

	\vspace{2mm}
	In this section, we show that the sequence $\mathcal{S}(N)$ enjoys the Hardy-Littlewood majorant property, under the following  assumption on the cardinality of $\mathcal{S}(N)$: there exist positive constants $c_1$ and $c_2$ such that
	\begin{equation} 
		\tag{H3}\label{hypo-3}
		\frac{c_1N}{G(z)}\leq |\mathcal{S}(N)|\leq \frac{c_2N}{G(z)}\; .
	\end{equation}
	
	\begin{theorem}
		Let $(a_n)_{n\leq N}$ be a sequence of complex numbers such that $\sum_{n\in\mathcal{S}(N)}|a_n|^2\leq \mathcal{S}(N)$. Suppose the Hypothesis \eqref{hypo-3} is satisfied.  Then for any $\ell\geq2$ we have
		\[\int_0^1\left|\sum_{n\in\mathcal{S}(N)}a_ne(n\alpha)\right|^\ell d\alpha\ll_\ell\int_0^1\left|\sum_{n\in\mathcal{S}(N)}e(n\alpha)\right|^\ell d\alpha\;.\]
	\end{theorem}
	\begin{proof}
		For $\alpha\leq \frac{c_1}{3c_2\pi N}$, by partial summation, we have 
		\begin{align*}
			\left|\sum_{n\in \mathcal{S}(N)}e(n\alpha)\right|&\geq |\mathcal{S}(N)|-\left|2\pi \alpha\int_{1}^N\mathcal{S}(t)e(t\alpha)dt\right|\\
			&\geq \frac{c_1N}{G(z)}-\frac{c_22\pi \alpha N^2}{G(z)}  \\
			&\geq \frac{c_1N}{3G(z)}\; .
		\end{align*}
		From this, we obtain
		\[\int_0^1\left|\sum_{n\in\mathcal{S}(N)}e(n\alpha)\right|^\ell d\alpha\geq\int_{\alpha\leq\frac{c_1}{c_23\pi N}}\left|\sum_{n\in\mathcal{S}(N)}e(n\alpha)\right|^\ell d\alpha\gg N^{\ell-1}/(G(z))^\ell\;.\] 
		Let $\ell>2.$ From  the above estimate together with Theorem~\ref{coro}, we get
		
		\[\int_0^1\left|\sum_{n\in\mathcal{S}(N)}a_ne(n\alpha)\right|^\ell d\alpha\ll_\ell \frac{N^{\ell-1}}{G(z)^\ell}\ll_{\ell}\int_0^1\left|\sum_{n\in\mathcal{S}(N)}e(n\alpha)\right|^\ell d\alpha\; .\]

		\noindent   The case $\ell=2$ is a consequence of Parseval's identity.
	\end{proof}
	
	\section{Applications}\label{appl}
	In this section, we provide some applications of our results.
	
	\vspace{2mm}
	\noindent\textit{Application-1}:  Let $F(x)=\prod_{i=1}^k (a_ix+b_i)$ be a polynomial with integer coefficients, $a_i\neq 0$, and  non-zero discriminant, i.e 
	\[\Delta(F):=\prod_{1\leq i<j\leq k}(a_ib_j-a_jb_i)\neq 0\;.\] 
	Let 
	\[
	X(F)=\{n\in \mathbb{N}: F(n) \text{ is the product of } k \text{ primes}\}
	\] and 
	\begin{equation}
		\label{gamma(q)}
		\gamma(q):=\frac{1}{q}\#\{n\in \mathbb{Z}/q\mathbb{Z}: (q, F(n))=1\}\; .
	\end{equation}
	
	With these notations, Green-Tao (\cite[Theorem~1.1]{Green-Tao*06}) proved the following theorem.
	
	\begin{theorem}
		\label{GT}
		Let $F$ be as defined above and suppose $\gamma(q)>0$ for all $q\geq 1$ and $|a_i|, |b_i|\leq N$. The for any $2<\ell<\infty$, we have 
		\[
		\int_{0}^1\left|\sum_{n\in [1,N]\cap X(F)}e(n\alpha)\right|^\ell d\alpha \ll_{\ell, k} G_F^\ell \frac{N^{\ell-1}}{\log^{k\ell} N}\; ,
		\]
		where $G_F=\prod_p \frac{\gamma(p)}{(1-\frac{1}{p})^k}$\; .
	\end{theorem}

	\vspace{2mm}
	\noindent
	Let $\calP=$ the set of all primes, let $z=N^{1/4}.$ Also let $\calL_p=\{-a_i^{-1}b_i\pmod p \text{ for } 1\leq i\leq k\}$. Observe that
	\begin{equation}
		\label{L_p and gamma(p)}
		\calL_p=\Z/p\Z\setminus\{n\in\Z/p\Z: \gcd(p,F(n))=1\}\;,
	\end{equation}
	and $\lambda(p)=k$ for all $p>k$ and $p\nmid \Delta$. By the observation~\eqref{L_p and gamma(p)}, for all primes $p$
	\begin{equation}
		\label{Our h and GTs' h}
		h(p)=\frac{\lambda(p)}{p-\lambda(p)}=\frac{1-\gamma(p)}{\gamma(p)}\;.
	\end{equation}
	It can be shown that $\mathcal{S}(N)$ contains all but at most $kz$ elements of $X(F)\cap[1,N]$.
	Note that the hypotheses~\eqref{H1} and~\eqref{H2} of Theorem~\ref{dual} are satisfied. Also, we have the following estimates (see~\cite{Green-Tao*06})
	\begin{equation}
		\label{lower bound of G(z)}
		G(z)\gg_k\frac{\log^kz}{G_F}\;.
	\end{equation}
	Using~\eqref{lower bound of G(z)} by Theorem~\ref{dual}, we get
	\begin{align*}
		\sum_{n\in \mathcal{S}(N)\cap X(F)}\left|\sum_{b\in \mathcal{B}}f(b)e(nb)\right|^2&\leq \sum_{n\in \mathcal{S}(N)}\left|\sum_{b\in \mathcal{B}}f(b)e(nb)\right|^2\\
		&\leq K\frac{N+\delta^{-1}}{G(z)}\log^k\left(2\frac{\|f\|_1^2}{\|f\|_2^2}\right)\|f\|_2^2\\
		&\ll (N+\delta^{-1})G_F\log^{-k}z\log^k\left(2\frac{\|f\|_1^2}{\|f\|_2^2}\right)\|f\|_2^2\; .
	\end{align*}
	\noindent
	From this result, following the proof of Theorem~\ref{coro}, we get
	\begin{align}
		\label{intersection bound}
		\int_0^1\left|\sum_{n\in X(F)\cap\mathcal{S}(N)}e(n\alpha)\right|^{\ell}d\alpha\ll_{l,k}& \frac{1}{N} \left(\frac{NG_F}{\log^k N}\sum_{n\in X(F)\cap\mathcal{S}(N)}1\right)^{\ell/2}\nonumber\\
		\ll_{k,\ell}&G_F^\ell\frac{N^{\ell-1}}{\log^{k\ell} N}\;,
	\end{align}
	here we used that 
	\[|\mathcal{S}(N)|\ll \frac{NG_F}{\log^k z},\]
	which is a consequence of the arithmetic large sieve inequality and~\eqref{lower bound of G(z)}.
	Then Theorem~\ref{GT} follows from~\eqref{intersection bound} and the fact that, $\mathcal{S}(N)$ contains all but at most $kz=kN^{1/4}$ elements of $X(F)\cap[1,N]$.

	\vspace{2mm}
	\noindent
	\textit{Application-2}:  Let \[S^\flat_z(\alpha;N)=4C\prod_{\substack{p\leq z\\ p\equiv 3\pmod 4}}\left(1-\frac{1}{p}\right)^{-1}\sum_{\substack{n\leq N\\ n\equiv 1\pmod 4\\ (n, P_{4,3}(z)=1)}}\frac{e(n\alpha)}{\sqrt{\log n}}\; .\]
	Here $C=\sqrt{2}\prod_{p\equiv3\pmod 4}\left(1-\frac{1}{p^2}\right)^{-1/2}$ and $P_{4,3}(z)=\prod_{\substack{p\equiv 3\pmod 4\\ p\leq z}}p.$ 
	We prove the following theorem as an application of Theorem~\ref{coro}.
	\begin{theorem}\label{L1b}
		For any $\ell>2$ and $N^{1/4}\geq z\ge2$, we have
		\begin{equation*}
			\int_0^1\biggl|S^\flat_z(\alpha;N)  \biggr|^\ell d\alpha\ll_{\ell}
			\frac{N^{\ell-1}}{(\log N)^{\ell/2}}.
		\end{equation*}
	\end{theorem}
	\noindent
	An upper bound is obtained in \cite{ORKV} (see Lemma 9.4) with $\sqrt{\log z}\frac{N^{\ell-1}}{(\log N)^{\ell/2}}$ on the right-hand side, and we have improved their bound as anticipated by the authors. 
	\begin{proof} 
		We have
		\[
		\int_0^1\biggl|S^\flat_z(\alpha;N)  \biggr|^\ell d\alpha=(4C)^\ell \prod_{\substack{p\leq z\\ p\equiv 3\pmod 4}}\left(1-\frac{1}{p}\right)^{-\ell} \int_{0}^1\left|\sum_{\substack{n\leq N\\ n\equiv 1\pmod 4\\ (n, P_{4,3}(z))=1}}\frac{e(n\alpha)}{\sqrt{\log n}}\right|^\ell d\alpha\; .
		\]
		Let $\calP$ be the set of all primes such that $p\equiv 3\pmod 4,$ and let $\calL_p=\{0\}$ for all $p\in\calP$. Then applying~\cite[Theorem A.3]{FI} we obtain
		\[G(z)\gg\prod_{\substack{p\leq z\\ p\equiv 3\pmod 4}}\left(1-\frac{1}{p}\right)^{-1}.\] We take $a_n=\frac{1}{\sqrt{\log n}}$ when $(n, P_{4,3}(z))=1$ and $n\equiv 1\pmod 4$, and $0$ otherwise in Theorem~\ref{coro}. This gives us 
		\begin{align*}
			\int_{0}^1|\sum_{\substack{n\leq N\\ n\equiv 1\pmod 4\\ (n, P_{4,3}(z)=1)}}\frac{e(n\alpha)}{\sqrt{\log n}}|^\ell d\alpha&\ll \frac{1}{N}\left(N\prod_{\substack{p\leq z\\ p\equiv 3\pmod 4}}\left(1-\frac{1}{p}\right)\sum_{\substack{n\leq N\\ n\equiv 1 \pmod 4\\ (n, P_{4,3}(z))=1}}\frac{1}{\log n} \right) ^{\ell/2}\\
			&\ll \frac{1}{N}\left(\frac{N^2}{\log N}\prod_{\substack{p\leq z\\ p\equiv 3\pmod 4}} \left(1-\frac{1}{p}\right)^2 \right) ^{\ell/2}\\
			&\ll \frac{N^{\ell-1}}{\log^{\ell/2} N}\prod_{\substack{p\leq z\\ p\equiv 3\pmod 4}}\left(1-\frac{1}{p}\right)^\ell\; .
		\end{align*}
		Hence we have 
		\[
		\int_0^1\biggl|S^\flat_z(\alpha;N)  \biggr|^\ell d\alpha\ll_{\ell} \frac{N^{\ell-1}}{\log^{\ell/2}N}\; .
		\]
	\end{proof}

\end{document}